%% file: whh514.tex
\documentclass{article}
\usepackage[english]{babel}
\usepackage{hyperref}
\hypersetup{%
    breaklinks=true, %draft,	% = no hyperlinking at all (useful in b/w printouts)
    colorlinks=true, linkcolor=blue,citecolor=blue%
    }
\usepackage[utf8]{inputenc}

\usepackage{verbatim}
\usepackage{amsmath}
\usepackage{amsthm}
\usepackage{graphicx}

\usepackage[numbers,sort&compress]{natbib}
\input latexmacrosG

\usepackage{microtype}
\usepackage{color}
\newtheorem{theorem}{Theorem}
\newtheorem{lemma}{Lemma}
\newtheorem{proposition}{Proposition}
\newtheorem{corollary}{Corollary}

\newtheoremstyle{examplestyle} % Name
            {}                   % Above skip 
            {}                   % Below skip
            {\upshape}    % Body font
            {}                   % Indent
            {\itshape}      % Head font
            {}                   % Head body p\\unct
            {0.5em}         % Space after head
            {}                   % Heading

\theoremstyle{examplestyle}
\newtheorem{example}{Example}
\newtheorem{remark}{Remark}

\DeclareMathOperator{\convexhull}{cv}
\DeclareMathOperator{\arcosh}{arcosh}

\newcommand{\parm}{\makebox[1ex]{$\cdot$}}

\makeatletter
\DeclareRobustCommand{\nscref}[1]{%
  \begingroup\@cref@sortfalse\cref{#1}\endgroup
}
\makeatother

\makeatletter
\newcommand{\Rmnum}[1]{\expandafter\@slowromancap\romannumeral #1@}
\makeatother

\title{Conditionally complete sponges: new results on generalized lattices}
\author{Jasper J. van de Gronde \and Wim H. Hesselink}

\begin{document}
\maketitle

\setcounter{tabnr}{-1}

\begin{abstract}
Sponges were recently proposed as a generalization of lattices, focussing on joins/meets of sets, while letting go of associativity/transitivity.
In this work we provide tools for characterizing and constructing sponges on metric spaces and groups.
These are then used in a characterization of epigraph sponges: a new class of sponges on Hilbert spaces whose sets of left/right bounds are formed by the epigraph of a rotationally symmetric function.
We also show that the so-called hyperbolic sponge generalizes to more than two dimensions.
\end{abstract}

%\tableofcontents

\section {Introduction}

Sponges are generalizations of lattices that were recently introduced by van de Gronde \cite{Gronde2015BSM,Gronde2015SGM,Gronde:2015:GMS}, as a possible solution to the long-standing problem of applying mathematical morphology to non-scalar data.
Morphological theory is based on lattices, but these are fairly restrictive when it comes to defining group-invariant instances on vector spaces (let alone manifolds) \cite{Birkhoff1961,Gronde2014GIC}, and are utterly incompatible with periodic spaces (assuming we wish to somehow preserve the periodicity in the lattice structure).
As a result, over the years several schemes have been suggested for morphological purposes that let go of lattices, or that try to work around the issue, while retaining something resembling a lattice's join and meet \cite{Peters1997,Zanoguera2002INS,Angulo2005MCP,Burgeth2007,Aptoula2009MPH,VelascoForero2011,Angulo2012CRP,Angulo2013SIN,Angulo2014MPU,Deborah2015SOA,Welk2015NAA}.
Unfortunately most of these schemes lack(ed) a supporting body of theory, making it hard to say much about the behaviour of the resulting filters. Sponges are meant to provide exactly such a framework, and have already been shown to encompass two schemes for vector spaces and hyperbolic spaces \cite{Zanoguera2002INS,Angulo2014MPU}, to allow the processing of angles in a natural way (without breaking the periodic nature of angles), and to support joins/meets on (hemi)spheres with a designated ``lowest'' point. Here we will provide additional examples, as well as tools to uncover further sponges.

Roughly speaking, an \emph{orientation} is a partial order without
transitivity and a \emph{sponge} is a set with an orientation that has meets
and joins for all subsets satisfying certain conditions.  The relevance of a meet or
join of a set in the absence of transitivity is due to preservation
under isometries (or other kinds of automorphisms) that permute the elements of the set.  Indeed,
long before the introduction of sponges, the meet of the inner-product
sponge was already used for this purpose \cite {H78,H79}.

As examples of sponges, van de Gronde and Roerdink \cite{Gronde:2015:GMS} present the
inner-product sponge \cite{Zanoguera2002INS}, the
one-dimensional angle sponge, and a two-dimensional hyperbolic sponge \cite{Angulo2014MPU}.  These examples are generalized and
treated in \nscref {ip_sponge,torus_sponge,hyp_sponge}, respectively.
Note that sponges on spheres and hemispheres were also given in \cite{Gronde:2015:GMS}, but these are (essentially) isomorphic to the inner-product sponge, so we do not treat these here.

In this work we first give a short overview of the main definitions concerning orientations and sponges in \cref{sec:definitions}. In \cref{ip_sponge} we briefly revisit the inner-product sponge \cite{Zanoguera2002INS,Gronde:2015:GMS}. Next, in \cref{sec:metricSponge} we derive a new result that makes it easier to identify sponges in metric spaces. In \cref{torus_sponge} we discuss sponge groups, and in \cref{sec:epigraphSponges} we introduce (and characterize) a new class of sponge (groups) called \emph{epigraph sponges}. In \cref{hyp_sponge} we generalize the hyperbolic sponge to the higher dimensional case (previously, only the 2D case was treated \cite{Angulo2014MPU,Gronde:2015:GMS}).

\section{Definitions}\label{sec:definitions}
Let $S$ be a set.  An \emph{orientation} of $S$ is a binary relation
$\preceq$ on $S$ that satisfies:
\begin{tab}
\> reflexivity: $x \preceq x$ \ for all $x\in S$, and\\
\> antisymmetry: $ x \preceq y\Land y\preceq x \implies x = y $ \ for all $x$, $y\in S$.
\end{tab}
The pair $(S, \preceq)$ of a set $S$ with an
orientation $\preceq$ is called an \emph{oriented set}.
A transitive orientation is a partial order.
If the
orientation $\preceq$ is not transitive, we may consider its
reflexive-transitive closure $\preceq^*$.  The orientation $\preceq$
is called \emph{acyclic} iff it contains no cycles, which is equivalent to $\preceq^*$ being a partial order.

If $P$ and $Q$ are subsets of an oriented set $(S, \preceq)$, we write
$P \preceq Q$ to denote that $p\preceq q$ holds for all $p\in P$ and
$q\in Q$.
A subset $P$ of $S$ is called \emph{right-bounded} iff
$P\preceq \{s\}$ for some $s\in S$; it is called \emph{left-bounded}
iff $ \{s\} \preceq P $ for some $s\in S$.\footnote{Left and right, rather than lower and upper, are used to warn the reader about the lack of transitivity.} Let the set of all right bounds
of $P$ be denoted by $R(P)$, and the set of all left bounds by $L(P)$. We abbreviate $R(\{x\})$ by $R(x)$.

Let $J(P)$ and $M(P)$ be subsets of $S$ defined by
\begin{align*}
x\in J(P) &\EQ P \preceq \{x\} \Land \left(\forall y\in S: P\preceq \{y\}
\implies x\preceq y\right) ,\\
x\in M(P) &\EQ \{x\} \preceq P \Land \left(\forall y\in S: \{y\}\preceq P
\implies y\preceq x\right).
\end{align*}
If $x$, $y\in J(P)$, then $P\preceq \{x\}$ and $P\preceq \{y\}$, and
hence $x\preceq y$ and $y\preceq x$, and therefore $x = y$ by
antisymmetry.  This proves that $J(P)$ is always empty or a singleton
set \cite{Fried1973WAL}.  A similar argument proves that $M(P)$ is always empty or a
singleton set.  If $J(P)$ or $M(P)$ has an element, its unique element
is called the \emph{join} or \emph{meet} of $P$, respectively.

A \emph{sponge} is defined to be an oriented set $(S, \preceq)$ in which every finite, nonempty, right-bounded subset has a join, and every finite, nonempty, left-bounded subset has a meet.
If the property holds for joins but not necessarily for meets, we have a join-semisponge (we can define a meet-semisponge analogously).
Note that in the original introduction of sponges, $J$ and $M$ were considered partial functions returning a particular element rather than a set of elements. Given that in an orientation $J$ and $M$ always return either the empty set or a singleton set, these views are equivalent.

Alternatively, a sponge can be defined algebraically as a set $S$ with functions $J$ and $M$, with a domain that includes (at least) all finite, nonempty subsets of $S$ and a range that includes \emph{no more} than all singleton subsets of $S$, as well as the empty set. To be a sponge, $J$ and $M$ should satisfy (with $y\in S$ and $P$ a finite, nonempty subset of $S$):
\begin{tab}
\> absorption: $\forall x\in P:M(\{x\}\cup J(P))=\{x\}$,\\
\> part preservation: $\!\begin{aligned}[t]&\left[\forall x\in P:M(\{x,y\})=\{y\}\right]\\&\!\!\implies M(P)\neq\emptyset \wedge M(M(P)\cup \{y\})=\{y\},\end{aligned}$\\
\> and the same properties with the roles of $J$ and $M$ reversed.
\end{tab}
Note that compared to the original algebraic definition \cite[\textsection4.2]{Gronde2015SGM}, absorption is now defined slightly more elegantly, and idempotence now follows from the two absorption laws:
\begin{equation*}
M(\{x\})=M(\{x\}\cup\{x\})=M(\{x\}\cup J(\{x\}\cup M(P)))=\{x\}
\end{equation*}
for any (finite) $P\supseteq\{x\}$ (the analogous statement $J(\{x\})=\{x\}$ also holds). In contrast, part preservation needs to explicitly claim that $M(P)$ is nonempty. These changes occur because the empty set behaves differently from the ``undefined'' value used in the original definition.
It has been shown \cite[\textsection4.3]{Gronde2015SGM} that the orientation-based and algebraic definitions are equivalent.

Compared to the algebraic definition of a lattice, the main difference is that we have the somewhat weaker property of part preservation rather than associativity. On the other hand, given that $J$ and $M$ operate on sets rather than being binary operators, commutativity is implied.
It should be noted that sponges are closely related to the concept of a weakly associative lattice (WAL) or trellis \cite{Skala1971TT,Fried1973NEC,Fried1973SEW}.
However, a WAL requires the join and meet to be defined for all pairs rather than all (finite and nonempty) \emph{bounded} sets.
It is known that the former by no means implies the latter \cite{Fried1973WAL}, and this makes WALs less suited for use in mathematical morphology, as this field relies heavily on the existence of joins and meets of sets. Conversely, in mathematical morphology it often suffices to guarantee the existence of joins and meets of bounded sets, again making sponges a better fit than WALs.
Often, it is convenient to be able to consider joins/meets not just over finite sets, but also infinite sets. This is part of our motivation to focus on conditionally complete sponges in the current work (the other part being that all practical examples examined so far belong to this category).

An oriented set $(S, \preceq)$ is called a \emph{conditionally complete sponge} (or \emph{cc sponge} for short) iff, for
every nonempty right-bounded subset $P$ of $ S$, the set $J(P)$ is
nonempty.
Note that, strictly speaking this is the definition of a conditionally complete join-\emph{semi}sponge, but \cref{thm:ccsponge} shows that a cc semisponge is also a cc sponge.

\begin{lemma}\label{thm:ccsponge}
  Let $(S, \preceq)$ be a cc sponge.  Let $P$ be a nonempty left-bounded
  subset of $S$.  Then $M(P)$ is nonempty.
\end{lemma}

\begin{proof}
  Define $Q = \{x \mid \{x\} \preceq P\}$, and choose some $p\in P$.
  Then $Q\preceq \{p\}$.  As $P$ is left bounded, $Q$ is nonempty.  As
  $(S, \preceq)$ is a sponge, $J(Q)$ is nonempty.  Therefore it
  suffices to prove that $J(Q) \subseteq M(P)$.  Let $x\in J(Q)$, that
  is $Q \preceq \{x\}$ and
\begin{equation*}
\forall y\in S: Q\preceq \{y\}\implies x\preceq y.
\end{equation*}
We need to prove $x\in M(P)$.  For every $y\in P$, we have $Q \preceq
\{y\}$ and hence $x\preceq y$; this proves $\{x\} \preceq P$.  Now let
$y\in S$ have $\{y\} \preceq P$.  Then $y\in Q$ and hence $y\preceq x$. 
This proves $x\in M(P)$. 
%\qed
\end{proof}

\begin{example}
The set $\IR$ of the real numbers with $\le$ as orientation is a cc sponge (a cc lattice in fact),
because every nonempty bounded subset of $\IR$ has a supremum.
\end{example}

\section {The inner-product sponge} \label {ip_sponge} 

Let $E$ be a real Hilbert space.  Let
relation $\preceq$ on $E$ be defined by \cite[\textsection5.1]{Gronde:2015:GMS}
\begin{equation*}
x\preceq y \EQ (x,x) \le (x, y) .
\end{equation*}
It is clear that $x \preceq x$ always holds.

\begin{example}
  Assume $E = \IR^2$ with the standard inner product.  Consider the
  four vectors $w = (1,0)$, $x = (2,0)$, $y = (2,1)$, and $z = (1,3)$.
  Then we have $w \preceq \{x,y,z\}$, and $x \preceq y$, and
  $y \preceq z$, but $x \not \preceq z$. It follows that,
  $x\in M(\{x,y\})$ and $y\in M(\{y, z\})$, but $x\notin M(\{x,y,z\})$. %\qed
\end{example}

\begin{lemma}
Let $ x\preceq y$ and $x \ne y$. Then $\|x\| < \|y\|$.
\end{lemma}

\begin{proof}
  If $x=0$, the assertion holds trivially. We may therefore assume
  that $x \ne 0$.  Therefore $\|x\| > 0$.  By Cauchy-Schwarz, $|(x,
  y)| \le \|x\|\cdot \|y\|$ with equality if and only if $y$ is a
  multiple of $x$.  On the other hand, $\|x\|^2 = (x,x) \le |(x, y)|$
  because $x \preceq y$.  It remains to consider the case that $y$ is
  a multiple of $x$, say $y=\lambda x$.  As $\|x\| > 0$ and $(x,x)
  \le (x,y)$, this implies $\lambda \ge 1$, and hence $y=x$ or
  $\|x\| < \|y\|$. 
\end{proof} 

\begin{corollary}
Relation $\preceq$ is an acyclic orientation on $E$. 
\end{corollary}

\begin{theorem}
The pair $(E, \preceq)$ is a cc sponge with a least element. 
\end{theorem}

\begin{proof}
It can be verified that $0$ is less than (or equal to) every element in $E$.
It therefore suffices to show that every nonempty subset of $E$ has a meet.
Let $P$ be a nonempty subset of $E$. It suffices to show that $P$ has
a meet.  We have 
\begin{equation*}
\{x\} \preceq P \EQ P \subseteq R(x),
.\end{equation*}
If $x = 0$ then $R(x)
= E$.  In all other cases, $R(x)$ is the closed halfspace $\{ y \mid (x,x) \le (x, y) \}$.  The
intersection of \emph{all} closed halfspaces that contain $P$ is the closed
convex hull $\convexhull(P)$ of $P$, i.e., the topological closure of the convex hull
of $P$.

We distinguish two cases.  First, assume there is no $x\ne 0$ with $\{x\}
\preceq P$. Then it is easily seen that $0$ is the meet of $P$. 
Otherwise, there exists $x\ne 0$ with $\{x\} \preceq P$.  Then all
elements of $\convexhull(P)$ are farther from the origin then $x$.  As
$\convexhull(P)$ is closed, convex, and nonempty in the Hilbert space $E$, there is a
unique point $z\in \convexhull(P)$ with smallest distance $\|z\|$ to the
origin.  We claim that $z$ is the meet of $P$.

We first prove $\{z\} \preceq P$. Indeed, for any $p\in P$, the line
segment between $z$ and $p$ is contained in $\convexhull(P)$; therefore all
its points have a distance to the origin $ \ge \|z\|$; therefore the
angle between the vectors $p-z$ and $0-z$ is not sharp, i.e. $(p-z,
0-z) \le 0$, and hence $(z,z) \le (z,p)$, i.e. $z \preceq p$.

It remains to observe that, for any vector $y$ with $\{y\} \preceq P$,
we have $z\in \convexhull(P) \subseteq R(y)$, so that $z\preceq y$.  This
proves that $z$ is the meet of $P$. 
\end{proof}

We remark that although every nonempty subset of $E$ has a meet, not every nonempty subset also has a join. In particular, the subset needs to be contained in a ball with the origin on its boundary to even be right bounded. It is, however, possible to extend $E$ with an extra element so that every nonempty set is right bounded \cite[\textsection5.3]{Gronde:2015:GMS}.

\section {A sponge in a complete metric space}\label{sec:metricSponge}

We consider a \emph{topological orientation} to be a topological space $S$ with an orientation, such that the orientation relation is a closed subset of the product space $S\times S$. In other words, if $\lim_{n\rightarrow\infty}x_n=x$ and $\lim_{n\rightarrow\infty}y_n=y$, and $x_n\preceq y_n$ for all $n\in\mathbb{N}$, then $x\preceq y$. This is  in line with the concept of a topological lattice used by Birkhoff \cite[\textsection\Rmnum{10}.11]{Birkhoff1995}, but slightly stricter than the analogous concept of a partially ordered topological space considered by Ward \cite{Ward1954POT} (who only requires sets of left and right bounds to be closed). Note that Birkhoff shows that the weaker concept is equivalent to the stronger concept in complete lattices; \cref{thm:closedBoundsClosedRelation} below shows that at least in some cases something similar holds for orientations as well. %\todo{Actually, I do not completely understand Birkhoff's proof, mostly because I am not entirely sure what ``the meet of the successors of $x_\alpha$'' means.}

Let $S$ be a complete metric space with distance function $d$.
Let $\preceq$ be a topological orientation on $S$.
Let a function $h: S \to \IR$ be called a \emph{discriminator} iff
\begin{equation*}
\forall\eps>0\;\exists\delta>0\;\forall x, y\in S:   x\preceq y \Land h(y) < h(x) + \delta \implies d(x,y) < \eps .
\end{equation*}
This condition implies that $h$ is strictly monotonic, in the sense
that $x\preceq y \Land x\ne y $ implies $h(x) < h(y) $.
The following theorem shows how in a complete metric space cc sponges can be characterized by the existence of meets of all left-bounded \emph{pairs} rather than all left-bounded nonempty sets. This is similar in spirit to what Birkhoff has shown for lattices \cite[\textsection\Rmnum{10}.10~Thm.~16]{Birkhoff1995}.

\begin{theorem} \label {metric2sponge} 
  Assume that $S$ is a complete metric space, that $(S,\preceq)$ is a topological orientation, that every left-bounded
  pair in $S$ has a meet, and that $h:S\rightarrow\IR$ is continuous and a discriminator.
  Then $(S, \preceq)$ is a cc sponge.
\end{theorem}

\begin{proof}
  Let $P$ be a nonempty right-bounded subset of $S$.  It suffices to
  prove that $P$ has a join.  Write $Q = \{x\mid P \preceq \{x\}\}$.
  As $P$ is right-bounded, $Q$ is nonempty.  For every $p\in P$,
  $q\in Q$, we have $h(p) \le h(q)$.  Every pair of elements of $Q$
  is left-bounded (by an element of $P$), and therefore has a meet,
  which is easily seen to be in $Q$. 

  Let $H$ be the infimum of $h(q)$ over all $q\in Q$, and $(q_n)_{n\in \Nat}$ an infinite
  sequence in $Q$ with $\lim _{n\to\infty} h(q_n) = H$.
  We first prove that this sequence
  is a Cauchy sequence.  Let $\eps > 0$ be given. 
  As $h$ is a discriminator, there is a number $\delta>0$ such that, for all $x,y\in Q$
  with $x\preceq y$ and $h(y) < h(x)+\delta$, $d(x,y)<\half\eps$.
  As $\lim _{n\to\infty} h(q_n) = H$, there
  is a number $m$ such that $h(q_n) < H + \delta$ for all $n \ge m$.
  For indices $i$, $j \ge m$, the pair $\{q_i, q_j\}$ in $Q$ has a
  meet $z_{ij}\in Q$.  Therefore, $H \le h(z_{ij})$.
  It follows that both $ h(q_i) $ and $h(q_j)$ are less than
  $h(z_{ij}) + \delta$.  As $\{ z_{ij} \} \preceq \{q_i, q_j\}$, this
  implies that $ d(q_i, q_j) \le d(q_i, z_{ij} ) + d(z_{ij}, q_j) < \eps $.
  This proves that $(q_n)_{n\in \Nat}$ is a Cauchy sequence.

  Because $S$ is a complete metric space, the Cauchy sequence has a
  limit, say $r$.  As function $h$ is continuous, $h(r) = H $.  On the
  other hand, $r\in Q$ holds because relation $\preceq$ is
  topologically closed.  For every $q\in Q$, the pair $q$, $r$ has a
  meet $z\in Q$, with $ h(r) = H \le h(z) $.  As $h$ is strictly
  monotonic and $z \preceq r$, it follows that $r = z \preceq q$. This
  proves $\{r \} \preceq Q$, and hence that $r$ is the join of $P$.
\end{proof}

Note that the above proof implies that $r$ in no way depends on the precise choice of the sequence $(q_n)_{n\in \Nat}$.
Also, it should be clear that we could just as easily have shown the dual statement, so for completeness:
\begin{corollary} \label {metric2spongeDual} 
  Assume that $S$ is a complete metric space, that $(S,\preceq)$ is a topological orientation, that every right-bounded
  pair in $S$ has a join, and that $h:S\rightarrow\IR$ is continuous and a discriminator.
  Then $(S, \preceq)$ is a cc sponge.
\end{corollary}

\begin{corollary} \label {meetIRsponge}
  Let $\preceq$ be a topological orientation on $\IR$, which
  implies $\le$.  Assume every left-bounded pair for $\preceq$ has a
  meet in $\IR$, and that the distance function is given by $d(x,y)=|x-y|$.  Then $(\IR, \preceq)$ is a cc sponge.
\end{corollary}

\begin{proof}
  \Cref {metric2sponge} is applied to $S := \IR$ with for $h$
  the identity function. 
  It is clear that $h$ is continuous. It is a discriminator because $x\le y<x+\eps$ implies $d(x,y)<\eps$.
\end{proof}

\section {Sponge groups}  \label {torus_sponge}

In this section, we investigate the possibility to combine the
structures of sponges and groups in a useful manner. Although it is
difficult to give interesting examples, we begin with not necessarily
commutative groups. In such a group, the group operation is denoted by
$\cdot$, and the neutral element by $\mathbf{1}$.
Note that some of the results shown here have been shown earlier for weakly associative lattices by Rach\r{u}nek \cite{Rachunek1979SOG}.\footnote{Rach\r{u}nek referred to orientations as semi-orders; unfortunately, this term is also used for other concepts.}

An \emph{oriented group} $G$ is defined to be a group 
with an orientation $\preceq$, such that
\begin{equation} \label {add_orient} 
\forall x,y,z\in G:\quad x \preceq y \implies x\cdot z \preceq y\cdot z \Land z\cdot x
  \preceq z\cdot y
.\end{equation}
It is called a \emph{cc sponge group} iff moreover $(G, \preceq)$ is a cc sponge.
Note that inversion in the group also reverses the order: $x\preceq y\EQ y^{-1}\cdot x\preceq \mathbf{1}\EQ y^{-1}\preceq x^{-1}$.

In an oriented group $(G, \preceq)$, with unit element $\mathbf{1}$, the
\emph{positive cone} is the subset $C$ of $G$ of the elements $x\in G$
with $\mathbf{1} \preceq x$.  It is easy to see that this set satisfies
\begin{equation} \label {cone_preceq}
\begin{aligned}
&C\cap C^{-1} = \{\mathbf{1}\},\\
&\forall x \in G, y \in C: x\cdot y \cdot x^{-1} \in C.
\end{aligned}
\end{equation}
The second condition says that $C$ is invariant under conjugation. 

Conversely, if $G$ is a group with a subset $C$ that satisfies the properties in \cref {cone_preceq}, then one can define the orientation $\preceq$ on $G$ by 
\begin{equation*}
x \preceq y \EQ x^{-1}\cdot y \in C .
\end{equation*}
This makes $(G, \preceq)$ an oriented group.  Indeed, relation
$\preceq$ is reflexive because $\mathbf{1} \in C$.  It is antisymmetric
because, if $ x\preceq y$ and $y\preceq x$, then $x^{-1}\cdot y\in
C\cap C^{-1} = \{\mathbf{1}\}$, so that $x = y$.  \cref {add_orient}
is easily seen to hold.  This proves that $(G, \preceq)$ is an
oriented group.  The orientation is a partial order if and only if
$C\cdot C \subseteq C$.

\begin{example}
Let $G = \mathrm{GL}_n(\IR)$, the group of the invertible real $n\times n$
matrices.  Let $C$ be the set of diagonalizable matrices with all
eigenvalues real and $\ge 1$.  The set $C$ satisfies the properties in \cref {cone_preceq}. It therefore induces an orientation $\preceq$
that makes $(G, \preceq)$ an oriented group. 

If $n > 1$, then $(G, \preceq)$ is not a sponge.  For $n = 2$, this
is shown as follows.  Assume that it is a sponge, and consider the
elements
\[ h(u) = \left ( \begin{array}{cc} 1 & u \\0 & 1 \end{array} \right) 
\qquad g(t) = \left ( \begin{array}{cc} 1 & 0 \\0 & t \end{array} \right) 
\]
For $ 1 < t$, we have $h(u) \preceq g(t) $ because
$h(u)^{-1} \cdot g(t)\in C$.  As $(G, \preceq)$ is a sponge, the pair
$h(0)$ and $h(1) $ has a join, say $k$.  For all $t > 1$, we have
$ \mathbf{1} = h(0) \preceq k \preceq g(t)$.  This implies that $k\in C$ and
$\det (k) = 1$.  The identity is the only element of $C$ with
determinant 1.  This proves that $k = \mathbf{1}$.  This implies
$h(1) \preceq \mathbf{1}$, a contradiction.  \qed
\end{example}

\begin{remark}
  At first glance, there do not appear to be many (interesting) noncommutative sponge groups.  The best
  one we found is the group $G$ of the real matrices
  \[ g(s,t,u) = \left( \begin{array}{cc} s & u \\0 & t \end{array}
  \right) \]
  with $s$, $t > 0$.  Let $C$ be the subset containing the matrices $g(1,1,u)$
  with $ u\ge 0$.  Then $(G, \preceq)$ is a sponge: a nonempty subset $V$ of
  $G$ is left bounded iff there are $s,t>0$ and $a\in\IR$ with $V\subseteq\{g(s,t,u)\mid a\le u\}$. Its meet is $g(s,t,b)$ for $b=\inf\{u\mid g(s,t,u)\in V\}$. This, however, is essentially just an additive
  sponge group embedded in a noncommutative group that preserves its
  orientation.
\end{remark}

\begin{lemma}\label{thm:closedBoundsClosedRelation}
Assume that $G$ is a topological group and that $(G,\preceq)$ is an oriented group (but not necessarily a topological orientation). Then, the positive cone $C$ is closed if and only if the relation $\preceq$ is closed.
\end{lemma}
\begin{proof}
If a function $g$ is continuous, the preimage of a closed set under $g$ is closed as well.
Now, note that $\preceq$ can be identified with the set $\{(x,y)\in G^2\mid x^{-1}\cdot y\in C\}$, the preimage of $C$ under the function $g_1(x,y)=x^{-1}\cdot y$. Owing to $G$ being a topological group, $g_1$ is continuous, and $\preceq$ is closed if $C$ is closed.
Next, note that $C=\{x\in G\mid \mathbf{1}\preceq x\}=\{x\in G\mid (\mathbf{1},x)\in{\preceq}\}$, the preimage of $\preceq$ under the continuous function $g_2(x)=(\mathbf{1},x)$. We thus also have that $C$ is closed if $\preceq$ is closed. This concludes the proof.
\end{proof}

\subsection {Refining the orientation}

For an oriented group $(G, \preceq)$ with a subset $C$, consider the
condition
\begin{equation} \label {defcone2} 
  y \in C \Land \mathbf{1} \preceq x \preceq y \implies x\in C \Land
  x^{-1}\cdot y \in C\text{ \ for all $y$, $x\in G$.}
\end{equation}

\begin{lemma} \label {restrSponge} Let $(G, \preceq)$ be a cc sponge
  group.  Let $C$ be a subset of $G$, invariant under conjugation,
  with $\mathbf{1} \in C$ and $\{\mathbf{1}\}\preceq C$.  Assume that \cref {defcone2} holds.  Let $\sqleq$ be the relation on $G$
  defined by $x\sqleq y\equiv x^{-1}\cdot y\in C$.  Then $(G, \sqleq)$
  is a cc sponge group.
\end{lemma} 

\begin{proof}
  The first formula of \cref {cone_preceq} holds because of $\mathbf{1} \in C$
  and $\{\mathbf{1}\}\preceq C$, and antisymmetry of $\preceq$.  The second one
  holds by assumption.  Therefore, $(G, \sqleq)$ is an oriented group.
  It remains to consider a nonempty subset $P$ of $G$ with $\{x\}
  \sqleq P$ for some $x$, and to prove that $P$ has a meet with
  respect to $\sqleq$.  By \cref {add_orient}, we may assume that $x
  = \mathbf{1}$.

  Assume $P$ is nonempty and satisfies $\{\mathbf{1}\} \sqleq P$.  It suffices
  to prove that $P$ has a meet for $\sqleq$.  By the definition of
  $\sqleq$, we have $P \subseteq C$.  As $(G, \preceq)$ is a cc sponge
  and $\mathbf{1} \preceq C$,
  the set $P$ has a meet for $\preceq$, say $y$.  We claim that $ y$
  is the meet of $P$ for $\sqleq$.

  To prove $\{ y \}\sqleq P$, let $p$ be an arbitrary element of $P$.
  Then $\mathbf{1} \preceq p$ and $ \mathbf{1} \preceq y \preceq p$ and $p\in C$.
  \Cref {defcone2} therefore implies that $ y^{-1}\cdot p \in
  C$, so that $ y \sqleq p$.  This proves $\{y\} \sqleq P$.

  For any $z$ with $\{ z\}\sqleq P$, we need to prove $ z \sqleq y $.
  The assumption $\{ z\}\sqleq P$ means that $z^{-1}\cdot P \subseteq
  C$.  For every $p\in P$, we therefore have $z^{-1}\cdot p\in C$, and
  hence $\mathbf{1} \preceq z^{-1}\cdot p$, and hence $z \preceq p$.  As $y$ is
  the meet of $P$ for $\preceq$, this implies $z\preceq y$.  It
  follows that $ \mathbf{1} \preceq z^{-1}\cdot y \preceq z^{-1}\cdot p \in C$
  for any $p\in P$.  Using the other part of \cref
  {defcone2}, we obtain $z^{-1}\cdot y \in C$ and hence $z \sqleq y$.
  This concludes the proof that $(G, \sqleq)$ is a cc sponge.
\end{proof}

\medbreak Condition (\ref {defcone2}) is sufficient but not
necessary. For instance, consider the additive group $\IR$ with
operation + and neutral element 0 (see \cref {add_sponge}).
Let $C$ be the subset
\begin{equation*} 
 C = \{ n + t \mid n\in \Nat  \land 0 \le t \le f(n) \} 
\end{equation*}
for some descending function $f: \Nat \to \IR$ with $0 \le f(0)$.
Let $\preceq$ be the associated orientation of $\IR$.  Then every
right-bounded pair has a join because
\begin{equation*}
 C \cap (y+C) \ne\emptyset \implies \exists z: z \in C \cap (y+C)
\subseteq z+C .
\end{equation*}
Using the automorphism $x \mapsto -x$, it follows that every
left-bounded pair has a meet.  Therefore, \cref
{meetIRsponge} implies that $(\IR, \preceq)$ is a sponge.
In fact, it is a sponge group.

\subsection {Quotient sets and factor groups} 

Let $(G,\preceq)$ be an oriented group and let $H$ be a subgroup of $G$.  Recall
that the (right) quotient set $G/H$ consists of the residue classes
$\overline {x} = x \cdot H$ for all $x\in G$.  The group $G$ has a
left action on the quotient $G/H$ defined by $g \cdot \overline{x} =
\overline {g\cdot x}$ for all $g\in G$.

Let $C$ be the positive cone of $(G,\preceq)$, and consider the relation $\sqsubseteq$ on $G/H$ defined by
\begin{equation*}
 \overline x \sqsubseteq \overline y \EQ x^{-1}\cdot y\in C\cdot H
,\end{equation*}
and the property
\begin{equation} \label {eq:antiH}
 \mathbf{1} \preceq q \Land \mathbf{1} \preceq r \Land q\cdot r \in H \implies q \in H \Land r \in H
\text{ \ for all $q$, $r\in G$.}
\end{equation}
Note that $x\preceq y\implies\overline{x}\sqsubseteq\overline{y}$, since $\mathbf{1}\in H$.

\begin{lemma}\label{thm:quotientOrientation}
Assume $(G,\preceq)$ is an oriented group, and $H$ a subgroup of $G$. Then $(G/H,\sqsubseteq)$ is an oriented set iff \cref{eq:antiH} is satisfied. If $H$ is a normal subgroup of $G$, it is an oriented group.
\end{lemma}
\begin{proof}
For $(G/H,\sqsubseteq)$ to be an oriented set, $\sqsubseteq$ must be reflexive and antisymmetric.
Now, since $\preceq$ is reflexive and $\mathbf{1}\in H$, $\sqsubseteq$ is reflexive as well. It remains to show that $\sqsubseteq$ is antisymmetric.
Assume $\overline x\sqsubseteq \overline y$ and
$\overline y \sqsubseteq \overline x$.  Then $x^{-1}\cdot y\in C\cdot H$ and
$y^{-1}\cdot x\in C\cdot H$.  Put $z = x^{-1}\cdot y$.  Then $z\in C\cdot H$ and
$z^{-1}\in C\cdot H$.  This implies that $H$ has elements $h$, $k$ with $h
\preceq z$ and $k\preceq z^{-1}$.  It follows that $1 \preceq h^{-1}z$
and $1 \preceq z^{-1}\cdot k^{-1}$.  As $h^{-1}\cdot z\cdot z^{-1}\cdot k^{-1}\in H$, \cref{eq:antiH} implies that $h^{-1}\cdot z \in H$.  It follows that $z\in
H$ and hence $\overline x = \overline y$.
As a result, $(G/H,\sqsubseteq)$ is an oriented set if $\cref{eq:antiH}$ holds.

Conversely, assume $(G/H,\sqsubseteq)$ is an oriented set. Assume there exist a $q$ and $r$ in $G$ such that $\mathbf{1}\preceq q$, $\mathbf{1}\preceq r$, and $q\cdot r\in H$.
Then $\overline{\mathbf{1}}\sqsubseteq \overline{q}$ because $\mathbf{1}\in H$.
On the other hand, $q\cdot r\in H$ and $\mathbf{1}\preceq r$ together imply that $q\preceq (q\cdot r)$, so that $\overline{q}\sqsubseteq \overline{\mathbf{1}}$. Since $\sqsubseteq$ is an orientation, we have $\overline{\mathbf{1}}=\overline{q}$, as well as $q\in H$ and $r\in H$.
So, $(G/H,\sqsubseteq)$ is an oriented set only if \cref{eq:antiH} holds.

Finally, if $H$ is a normal subgroup of $G$, $G/H$ is a group with $\overline{x\cdot y}=\overline{x}\cdot\overline{y}$, so if \cref{add_orient} holds, $(G/H,\sqsubseteq)$ is an oriented group.
Now, assume $\overline x\sqsubseteq\overline y$.
Then $x^{-1}\cdot y\in C\cdot H$, so $x^{-1}\cdot z^{-1}\cdot z\cdot y\in C\cdot H$ as well: $\overline{z}\cdot\overline{x}\sqsubseteq \overline{z}\cdot\overline{y}$.
Similarly, since $(G,\preceq)$ is an oriented group, we can use the second property of \cref{cone_preceq} together with the normality of $H$ to see that $z^{-1}\cdot x^{-1}\cdot y\cdot z\in C\cdot H$, and thus $\overline{x}\cdot\overline{z}\sqsubseteq \overline{y}\cdot\overline{z}$.
\end{proof}

We now give a sufficient condition for the orientation on $G/H$ to also be a cc sponge:
\begin{equation}\label{eq:quotientPostulate}
\forall z\in G: \exists h\in H: R(z)\cap C\cdot H \subseteq R(h)
.\end{equation}

\begin{lemma} Let $(G, \preceq)$ be a cc sponge group.  Let $H$ be a
  subgroup of $G$ that satisfies \cref{eq:antiH,eq:quotientPostulate}.
  Then $(G/H, \preceq)$ is a cc sponge. If $H$ is a normal subgroup of $G$, it is a cc sponge group.
\end{lemma}

\begin{proof} 
  It suffices to prove that every nonempty left-bounded subset $P$ of
  $G/H$ has a meet. By translation invariance, we may assume that the
  left bound of $P$ is $\overline{\mathbf{1}}$.  So, we have $\{\overline{\mathbf{1}}\}
  \sqsubseteq P$.  This implies that $G$ has a nonempty subset $Q$ with
  $\{\mathbf{1}\} \preceq Q$ and $P = \{\overline q \mid q \in Q\}$.  Now $Q$
  has a meet, say $m\in G$.  It satisfies $m \preceq q$ for all $q\in
  Q$.  Therefore $\overline m \sqsubseteq p$ for all $p\in P$.

  Moreover, let $y\in G$ be such that $\{\overline y\} \sqsubseteq P$.
  Then $\overline y \sqsubseteq \overline q$ for all $q\in Q$.  This
  implies that $y^{-1}\cdot Q \subseteq C\cdot H$.  On the other hand, by translation invariance, $y^{-1}\cdot Q
  \subseteq R(y^{-1})$.  \Cref{eq:quotientPostulate} with $z := y^{-1}$ now
  implies that $H$ has an element $h$ with $y^{-1}\cdot Q \subseteq R(h)$.
  It follows that $\{y\cdot h\} \preceq Q$.  As $m$ is the meet of $Q$, this
  implies $y\cdot h \preceq m$ and hence $\overline y \sqsubseteq \overline m$.
  This proves that $\overline m$ is the meet of $P$ in $G/H$.
  
  Finally, if $H$ is a normal subgroup of $G$, $(G/H,\preceq)$ is an oriented group by \cref{thm:quotientOrientation}. Since we have just shown it is also a cc sponge, it is a cc sponge group.
\end{proof}

\subsection {Additive sponges}  \label {add_sponge}

As we have found no interesting noncommutative sponge groups, we
henceforth restrict the attention to commutative sponge groups.  These
are written additively, with neutral element $\mathbf{0}$.  By the
commutativity of the group operation ($+$), \cref
{add_orient} reduces to 
\begin{equation*} 
 x \preceq y \implies x+ z \preceq y+ z  \tag{\ref {add_orient}'}\label{add_orient_comm}
.\end{equation*}
The positive cone $C$ is characterized by
\begin{equation*} 
 C\cap -C = \{\mathbf{0}\} . \tag{\ref {cone_preceq}'}\label{cone_preceq_comm}
\end{equation*}
So $(G,\preceq)$ is an oriented additive group if and only if there is a set $C$ that satisfies \cref{cone_preceq_comm}.

\begin{example}
 Take for the group the additive group
  $\IR^2$ with the orientation given by
\begin{tab} 
\> $(x_1, x_2) \preceq (y_1, y_2) \;\equiv \;x_1 \le y_1<x_1+\half \Land 
x_2 \le y_2<x_2+\half $ .
\end{tab}
Use \cref{restrSponge} to prove that this is a cc sponge group.  \Cref{eq:antiH,eq:quotientPostulate}
 hold for the grid $H = \IZ^2$.  They also hold if $H$ is one of
the two coordinate axes.  \Cref{eq:antiH} fails if $H$ is the line given by
$x_1 = x_2$.  If $H$ is the line $x_1 + x_2 = 0$, \cref{eq:antiH} holds
and \cref{eq:quotientPostulate} fails.  In this case $G/H$ is a sponge, but the
projection $G\to G/H$ does not preserve meets.
\end{example}

\section {Epigraph sponges} \label{sec:epigraphSponges}

  Let $E$ be a real Hilbert space with inner product $(\_\,,\_)$. We assume that $\dim(E)\ge2$.  Let
  $h$ be a unit vector in $E$.  For any $x\in E$, we can write $x=x_h h+x_\perp$, where $x_\perp$ is orthogonal to $h$ and $x_h\in\IR$. Note that $x_h=(x,h)$, because $x_\perp$ is orthogonal to $h$ and because $h$ is a unit vector.
  For a function $f:\IR_{\ge0}\rightarrow\IR_{\ge0}$, let
  $C_f$ be the subset of $E$ given by
\begin{equation*}
x\in C_f\iff f(\|x_\perp\|)\le x_h
.\end{equation*}
$C_f$ is the set of points on or above the graph (the \emph{epigraph}) of $f\circ \|\parm\|$ evaluated on the hyperplane through the origin perpendicular to $h$. %\todo{I have maintained the notation $x_h$ and $x_\perp$, since although $x_h$ is still reasonably clear when expressed in terms of inner products, $x_\perp=x-(x,h)h$, which gets old very quickly.}

\begin{proposition}\label{thm:epigraphOrientation}
$(E,\preceq_f)$, with $x\preceq_f y\EQ y-x\in C_f$, is an oriented group if and only if $f(d)=0$ and $f(d)>0$ for all $d>0$.
\end{proposition}
\begin{proof}
It is not too difficult to see that $C_f$ satisfies \cref{cone_preceq_comm} if and only if $f(0)=0$ and $f(d)>0$. This concludes the proof.
\end{proof}

From now on, by convention, we assume that $f(d)=0$ iff $d=0$. We thus have an oriented group $(E,\preceq_f)$ with
\begin{equation*}
x\preceq_f y\EQ y-x\in C_f\EQ f(\|y_\perp-x_\perp\|)\le y_h-x_h
.\end{equation*}

Before introducing the main theorem of this section, recall that a function $f:\IR_{\ge0}\rightarrow\IR_{\ge0}$ is called \emph{superadditive} iff $f(x+y)\ge f(x)+f(y)$ for all $x$, $y$. We define $f$ to be \emph{square-superadditive} iff $f(\sqrt{x^2+y^2})\ge f(x)+f(y)$ for all $x$, $y$. Note that $f$ is square-superadditive iff the function $\phi(x)=f(\sqrt{x})$ is superadditive.
Furthermore, if $f$ is superadditive or square-superadditive, it is also ascending (due to the nonnegativity of $f$), or increasing if $f$ is positive for all nonzero arguments.
Finally, if $f$ is square-superadditive, then $f$ is superadditive, due to $\sqrt{x^2+y^2}\le x+y$ for all nonnegative $x$ and $y$, and the ascendingness of square-superadditive functions.

We are now ready to state the main theorem of this section.

\begin{theorem} \label {thm:epigraph_sponges} 
  Assume that $f: \IR_{\ge0} \to \IR_{\ge0}$ satisfies $f(d) > 0$ iff $d>0$.
  \\(a) Let $ \dim(E) \ge 3$. Then $(E, \preceq_f)$ is a cc sponge if
  and only if relation $\preceq_f$ is topologically closed and $f$ is square-superadditive.
  \\(b) Let $\dim(E) = 2$ and relation $\preceq_f$ be topologically closed. Then $(E,\preceq_f)$ is a cc sponge if and only if $f$ is superadditive.
\end{theorem}
We conjecture that every epigraph sponge on a two-dimensional
space has a topologically closed relation $\preceq_f$.

\subsection{Properties of the oriented group}
As a preparation of the proof of \cref{thm:epigraph_sponges}, we investigate the oriented group $(E,\preceq_f)$ introduced in \cref{thm:epigraphOrientation}.

\begin{lemma} \label {thm:semicontinuity}
  Relation $\preceq_f$ is topologically closed in $E^2$ if and only if
  the function $f: \IR_{\ge0} \to \IR_{\ge0}$ is lower semicontinuous.
\end{lemma}
\begin{proof}
As $(E, \preceq_f)$ is an oriented group, $\preceq_f$ is closed if and only
if the epigraph $C_f = \{ w\mid f(\|w_\perp\|) \le (w, h) \}$ is
closed (\cref{thm:closedBoundsClosedRelation}).
By convention, $\dim(E) \ge 2$.  We can therefore choose a unit
vector $u $ orthogonal to $h$. 
Recall that $f$ is lower semicontinuous iff, for every $d$,
$a\in\IR_{\ge0}$ with $f(d) > a$, there exists $\eps > 0$ such that for all
$e\in\IR_{\ge0}$ with $|e-d| < \eps$ it holds that $f(e) > a$.

Assume that $C_f$ is closed, and that $f(d) > a$. Then the vector
$w = d\,u + a\,h $, is not in $C_f$.  As $C_f$ is closed,
there exists $\eps > 0$ such that the ball around $w$ with radius $\eps$
does not meet $C_f$. It follows that $f(e) > a$ for all real numbers
$e$ with $|e-d| < \eps$.  This proves that $f$ is lower semicontinuous.

Conversely, assume that $f$ is lower semicontinuous.  To show that
$C_f$ is closed, consider $w\notin C_f$.  This means that
$(w, h) < f(\|w_\perp\|)$. Choose a number $a$ with
$(w,h) < a < f(\|w_\perp\|) $.  As $f$ is lower semicontinuous, there
is a number $\eps > 0 $ such that $f(x) > a$ for all numbers $x$ with
$|x - \|w_\perp\|| < \eps$.  Furthermore, the point $w$ has an open neighborhood $N$
in $E$ such that all points $w'\in N$ satisfy $(w', h) < a$ and
$| \|w'_\perp\| - \|w_\perp\| | < \eps$. As a result,
$ (w',h) < f(\|w'_\perp\|)$, and hence $w'\notin C_f$.  This proves
that $C_f$ is topologically closed.
\end{proof}

\begin{proposition} \label {thm:finite_has_right_bounds} Every finite subset $P$
  of $E$ is left- and right-bounded in $(E, \preceq_f)$.
\end{proposition}
\begin{proof}
  For real $t$, the vector $t\, h$ is a right bound of $p\in P$ if
  and only if $f(\|p_\perp\|) + p_h\le t$.  Therefore, $t\, h$
  is a right bound of $P$ when $t$ is larger than the maximum of the
  numbers $f(\|p_\perp\|) + p_h$ with $p$ ranging over $P$. A left bound can be constructed analogously.
\end{proof}

We next observe that, owing to $f(d)>0$ for $d\neq0$, relation $\preceq_f$ satisfies
\begin{equation*}
x\preceq_f y \Land x \ne y \implies (x,h) < (y, h).
\end{equation*}
Recalling that $f(0)=0$, the above directly implies:

\begin{proposition} \label {thm:joinIsLowestPoint}
  Let $P$ be a subset of $E$ that has a join $x$.
  Then $x$ is the unique lowest point with respect to $h$ of the set
  of right bounds of $P$, i.e. $(x,h) < (y,h)$ for all right bounds
  $y\ne x$ of $P$.
\end{proposition}

\begin{proposition} \label {thm:join_in_subspace}
  Let $V$ be a finite-dimensional linear subspace of $E$ that
  contains $h$.  Let $P$ be a subset of $V$ that has a join (or meet)
  $x$.  Then $x\in V$.
\end{proposition}

\begin{proof} 
  As $V$ is finite-dimensional, the space $E$ is the direct sum
  $V\oplus V^\perp$.  Let $\zeta: E\to E$ be the linear mapping given by
  $\zeta(v+w) = v-w$ for all $v\in V$, $w\in V^\perp$. Then $\zeta$ is an
  isometry of $E$ which preserves $h$. It therefore preserves
  $\preceq_f$, and joins and meets for $\preceq_f$.  Hence, it keeps $x$
  invariant because it keeps $P$ invariant.  Finally, $\zeta(x) = x$
  implies $x\in V$.
\end{proof}

\begin{proposition} \label {thm:join_symmetry}
  Let $x$ and $y$ be two points in $E$, satisfying $x_h=y_h$ and $x_\perp=-y_\perp$.
  Then, if $x$ and $y$ have a join, it is the point $(x_h+f(\|x_\perp\|))h$.
\end{proposition}
\begin{proof}
  By \cref {thm:join_in_subspace}, if $x$ and $y$ have a join, it is an element of the subspace
  spanned by $x$, $y$ and $h$. We can also see that the join of $x$ and $y$ has to be a multiple of h.
  This is trivially true if $\dim(E)=1$. For $\dim(E)\ge2$ it is also true, as otherwise the symmetry of the problem would imply the existence of two equally valid candidates, contradicting \cref{thm:joinIsLowestPoint}.
  For real $t$, the multiple $t\, h$ is a right bound of $\{x,y\}$ if and
  only if $t \ge x_h+f(\|x_\perp\|)= y_h+f(\|y_\perp\|)$.  Therefore, by \cref
  {thm:joinIsLowestPoint}, the join, if it exists, is indeed equal to $(x_h+f(x)) h$.
\end{proof}

\subsection{Properties of epigraph sponges}
Having looked at some of the properties of the oriented group $(E,\preceq_f)$, we now consider what happens if it is in fact a sponge.

\begin{lemma}\label{thm:maxSuper}
Assume that $(E, \preceq_f)$ is a sponge. Then $f(d) + f(e) \le \max(f(d+e), f(|d-e|))$ for all $d, e \in \IR_{\ge0}$. If $\dim(E)\ge3$, then $f$ is square-superadditive.
\end{lemma}
\begin{proof} As $\dim(E)\ge 2$, we can choose a unit vector $u\in E$,
  orthogonal to $h$.   Let $d$ and $e$ be given. 
  Consider the doubleton set $P = \{d\, u, -d\, u\}$ in $E$.  By
  \cref {thm:finite_has_right_bounds}, the set $P$ has a right
  bound.  As $(E, \preceq_f)$ is a sponge, $P$ has a join. By \cref
  {thm:join_symmetry}, the join is $f(d)\,h$.  This implies
  that
\begin{equation} \label{join2}
\forall w\in E: \quad d\, u \preceq_f w \Land -d\, u \preceq_f w \implies f(d) \, h\preceq_f w.
\end{equation}
Applying \cref {join2} to $w = e\, u + a\, h $ (for
arbitrary $a\in \IR$), we observe:
\begin{align*}
&\forall a: \quad d\, u \preceq_f e\, u + a\, h 
\Land -d\, u \preceq_f e\, u + a\, h \implies f(d)\, h \preceq_f e\, u + a\, h \\
\equiv\quad&\forall a:\quad f(|d-e|) \le a \Land  f(d+e) \le a 
\implies f(e) \le a - f(d) \\
\equiv\quad& \forall a:\quad \max(f(d+e), f(|d-e|)) \le a \implies f(d) + f(e) \le a\\
\equiv\quad& f(d) + f(e) \le \max(f(d+e), f(|d-e|)) .
\end{align*}
This concludes the first part of the proof.

Now, if $\dim (E) \ge 3$, we can choose a unit vector $v$ orthogonal
to both $h$ and $u$. \Cref {join2} is now applied to $w =
e\,v + a\,h $ for arbitrary $a\in \IR$, and we have
\begin{align*}
&\forall a:\quad d\, u \preceq_f e\, v + a\, h 
  \Land -d\, u \preceq_f e\, v + a\, h\implies f(d) \, h \preceq_f e\, v + a\, h\\
\equiv\quad&\forall a: \quad f(\sqrt{d^2 + e^2}) \le a \implies f(e) \le a - f(d)\\
\equiv\quad&f(d) + f(e) \le f(\sqrt{d^2 + e^2}) 
.\end{align*}
This proves that $f$ is square-superadditive. 
This concludes the proof.
\end{proof}

\begin{proposition}\label{thm:ascending_implies_increasing}
Assume that $(E, \preceq_f)$ is a sponge. Then $f$ is ascending if and only if it is increasing.
\end{proposition}
\begin{proof}

If $f$ is ascending and not increasing, then there is some interval $[d_1,d_2]$ over which $f$ has a constant value, say $w$.
Choose $e$ with $0<e<\frac{1}{2}(d_2-d_1)$. Then $f(e)>0$ and $f(d_2)=f(d_2-e)=f(d_2-2e)=w$, so that \cref{thm:maxSuper} applied to $d_2-e$ and $e$ gives \begin{equation*}w<f(d_2-e)+f(e)\le\max(f(d_2),f(d_2-2e))=w.\end{equation*} This is clearly a contradiction, so we conclude that if $f$ is ascending, it must also be increasing.
\end{proof}

Recall that $f$ is lower semicontinuous iff, for every $d$,
$a\in\IR_{\ge0}$ with $f(d) > a$, there exists $\eps > 0$ such that for all
$e\in\IR_{\ge0}$ with $|e-d| < \eps$ it holds that $f(e) > a$.
Now, assume that $f$ is ascending.
Then all discontinuities of $f$ are ``of the first kind'' (jump discontinuities) \cite[Corollary to Thm.\@ 4.29]{Rudin1976PMA}.
That is, even if $f$ is discontinuous in $d$, the limits $f^-(d)=\lim_{e\uparrow d}f(e)$ and $f^+(d)=\lim_{e\downarrow d}f(e)$ exist, and $f^-(d)\le f(d)\le f^+(d)$.
Function $f$ is lower semicontinuous if and only if $f^-(d)=f(d)$ for all $d$.

\begin{lemma}\label{thm:ImpliesLowerSemicontinuous}
Assume that $(E, \preceq_f)$ is a sponge. Then $f$ is ascending if and only if it is lower semicontinuous.
\end{lemma}
\begin{proof}

First assume that $f$ is ascending. Let $d$ be an argument where $f$ is not continuous. As before, choose a unit vector $u$ orthogonal to $h$. Let the vectors $x$ and $y$ be given by $x=d\,u+f^+(d)\,h$ and $y=-d\,u+f^-(d)\,h$.
For real numbers $z_h$ and $e$, the vector $z=e\,u+z_h\,h$ is a right bound of $\{x,y\}$ if and only if \begin{equation*}z_h\ge r_e=\max(f(|d-e|)+f^+(d), f(|d+e|)+f^-(d)).\end{equation*}
By \cref{thm:join_in_subspace,thm:joinIsLowestPoint}, the join of $\{x,y\}$ is the lowest such right bound, so it is $z^*=e\,u+r_e\,h$ with $e=\arg\min_e r_e$.
Now note that for $e<0$, $f^+(d)<f(|d-e|)$, so that \begin{gather*}f^+(d)+f^-(d)<2\,f^+(d)<r_e\quad\text{for all $e<0$}.
\intertext{For $e>0$, $f^+(d)<f(|d+e|)$, so that}f^+(d)+f^-(d)<r_e\quad\text{for all $e>0$}.\end{gather*}
Furthermore, for $0<e<2d$, $f(|d-e|)<f^-(d)$, so that $r_e=f(d+e)+f^-(d)$.
It follows that $\lim_{e\downarrow0}=f^+(d)+f^-(d)$.
Clearly, this limit must be the height of the lowest right bound $z^*$ of $\{x,y\}$.
Now, since $r_e=f(d)+f^+(d)$ for $e=0$, the existence of the join implies that $f(d)=f^-(d)$. This proves that $f$ is lower semicontinuous.

  Now, assume that $f$ is lower semicontinuous, and that $f$ is \emph{not ascending}. So, there should be real numbers $u$ and $v$ such that
  $0 \le u < v$ and $f(v) < f(u)$.  As $f$ is lower semicontinuous, the set
  $G=\{ d\in \IR_{\ge0} \mid f(d) \le f(v) \} $ is topologically closed.
  It follows that its subsets $G_0 = \{d\in G \mid d \le u\}$ and
  $G_1 = \{d\in G \mid u \le d\}$ are also closed.  $G_0$ is bounded
  from above by $u$ and $G_1$ is bounded from below by $u$.  Therefore,
  $G_0$ has a greatest element $g_0$, and $G_1$ has a smallest
  element $g_1$.  It is clear that $g_0 < u < g_1$, and that
  $d = \half(g_0+g_1) \notin G$, so that $f(d)>f(v)$. Putting $e = \half(g_1-g_0)$, we have
  $0 < e \le d$ and $ \max(f(d+e), f(d-e))=\max(f(g_1), f(g_0)) \le f(v) < f(d)$.
  This contradicts \cref{thm:maxSuper}, so if $f$ is lower semicontinuous it is also ascending, completing the proof.
\end{proof}

Combining the above lemmas, we find the following:

\begin{corollary}\label{thm:equivalentProperties}\begin{samepage}
Assume that $(E, \preceq_f)$ is a sponge. Then the following are equivalent:
\begin{enumerate}\itemsep0em
\item relation $\preceq_f$ is topologically closed,
\item $f$ is lower semicontinuous,
\item $f$ is ascending,
\item $f$ is increasing,
\item $f$ is superadditive.
\end{enumerate}
If $\dim(E)\ge3$, all of the aforementioned properties hold.\end{samepage}
\end{corollary}
\begin{proof}
\Cref{thm:semicontinuity} shows that the first two properties are equivalent (even if $(E,\preceq_f)$ is just an orientation).
\Cref{thm:ascending_implies_increasing} shows that in the current context the third and fourth property are equivalent.
\Cref{thm:ImpliesLowerSemicontinuous} shows that the second and third property are equivalent.
We also noted already that if $f$ is superadditive, it is also ascending.
This leaves only one implication to prove: that if $f$ is ascending, it is also superadditive.
  If $f$ is ascending, \cref{thm:maxSuper} now implies that
  $f(d) + f(e) \le f(d+e)$ for all nonnegative reals $d$ and $e$, since those satisfy $|d-e|\le d+e$.
  This implies that $f$ is superadditive.
Finally, when $\dim(E)\ge3$, \cref{thm:maxSuper} tells us that $f$ is square-superadditive, and thus superadditive.
\end{proof}

\subsection{Sufficiency}
The only-if parts of \cref{thm:epigraph_sponges} are now proved as
follows: if $(E,\preceq_f)$ is a sponge and $\dim (E) \ge 3$, then
relation $\preceq_f$ is topologically closed and $f$ is
square-superadditive by \cref{thm:equivalentProperties}. If $(E,\preceq_f)$ is a
sponge, $\dim (E) = 2$, and $\preceq_f$ is topologically closed, then
$f$ is superadditive by \cref{thm:equivalentProperties}.
We now show that we can apply \cref{metric2spongeDual} for the if parts.
Consequently, in this section we will assume that $\preceq_f$ is topologically closed, and that $f$ is square-superadditive (or just superadditive if $\dim(E)=2$).
As a result, we may also assume that $f$ is increasing (since it is superadditive, and we assumed earlier that $f(d)>0$ for all $d>0$), and that $f$ is lower semicontinuous (\cref{thm:semicontinuity}).

\begin{lemma}\label{thm:hDiscriminator}
The covector $h^*:E\rightarrow\IR$, defined by $h^*(x)=(h,x)$, is continuous and a discriminator.
\end{lemma}
\begin{proof}
Being a linear bounded functional, $h^*$ is continuous \cite[Thm.\@ 1.18]{Rudin1991FA}.
To see that it is also a discriminator, consider $\eps>0$ to be given.
We can now pick a $\delta>0$ that is both less than $\half\eps$ and less than $f(\half\eps)$.
Clearly, recalling that $f$ is increasing, any element $x\in C_f$ for which $h^*(x)=x_h<\delta$ satisfies $\|x\|<\half\eps+\half\eps=\eps$.
Since $y_h-x_h<\delta\equiv h^*(y)<h^*(x)+\delta$, it follows that for any two elements $x,y\in E$, $x\preceq y$ and $h^*(y)<h^*(x)+\delta$ imply that $d(x,y)=\|y-x\|<\eps$.
\end{proof}

Recall that function $f$ is lower semicontinuous and increasing. It is easy to see that
$f(x) \le f^+(x) < f(y)$ whenever $0 \le x < y$.  We also have 
\begin{equation} \label {superplus}
f^+(x) + f(y) \le f(x+y) \text{ \ whenever } x\ge 0\text{ and }y > 0.
\end{equation}
This follows from the fact that $f$ is superadditive, increasing, and lower semicontinuous, as well as the fact that $f(x+\eps) + f(y-\eps) \le f(x+y)$ holds for all $\eps$ with $0 < \eps < y$.

\begin{proposition} \label {norms_superplus}  Let $V$ be a Hilbert space, and let
  $p$, $q\in V$ be such that $(p, q) \ge 0$ and $q \ne 0$.
Assume that $f$ is square-superadditive, or that $f$ is super-additive
and $\dim(V) = 1$.  Then $f^+(\|p\|) + f(\|q\|) \le f(\|p+q\|)$. 
\end{proposition}

\begin{proof} 
  First assume that $f$ is square-superadditive.  Let $\phi$ be the
  superadditive function given by $\phi(x) = f(\sqrt{x})$.  \Cref {superplus} implies that
  \begin{equation*} f^+(\|p\|) + f(\|q\|) = \phi^+(\|p\|^2) + \phi(\|q\|^2) \le
  \phi(\|p\|^2 + \|q\|^2) .\end{equation*}
  On the other hand, we
  have $ \phi(\|p\|^2 + \|q\|^2) \le \phi(\|p+q\|^2) = f(\|p+q\|)$, as $(p,q) \ge 0$, and $\phi$ is increasing.
  
  If $\dim(V) = 1$, we have $\|p+q\| = \|p\| + \|q\|$ because
  $(p,q) \ge 0$.  If, moreover, $f$ is superadditive, then \cref {superplus} gives
  $f^+(\|p\|) + f(\|q\|) \le f(\|p\|+\|q\|) = f(\|p+q\|)$.
\end{proof}

\begin{lemma}\label{thm:everyPairAJoin}
  Every pair of elements of $E$ has a join in $E$.
\end{lemma}

\begin {proof} 
  If $x$ and $y$ are comparable by $\preceq_f$, one of them is their
  join.  We may therefore assume that they are not comparable.
  Therefore, the difference vector $x-y$ is not a multiple of $h$.  We
  may translate the origin in the hyperplane $h^\perp$ of vectors orthogonal to $h$ to the point
  $\half(x_\perp+y_\perp)$, and thus assume that $x_\perp = -y_\perp\ne 0$.

  Let $e$ be the unit vector $\|y_\perp\|^{-1}\,y_\perp $.  Let $S$ be
  the linear subspace spanned by $h$ and $e$.  This subspace contains
  $x$ and $y$.  The vectors $h$ and $e$ form an orthonormal basis of
  it.  We abbreviate the inner products with $h$ and $e$ by
  $u_h = (u, h)$ and $u_e = (u, e)$.  Note that $x_e < 0 < y_e$.

  Let $U$ be the set of right-bounds of $\{x, y\}$, so that $U = (x +
  C_f)\cap (y + C_f)$.  As $x$ and $y$ are not comparable, we have
  $x\notin U$ and $y\notin U$.  Given the assumptions made at the start of this section, the set $U$ is topologically closed.
We observe that
\begin{align}
\notag  &\quad u\in U \\
\notag  \eq&\quad x\preceq_f u \Land y\preceq_f u \\
\notag  \eq&\quad f(\|u_\perp-x_\perp\|) \le u_h - x_h \Land  f(\|u_\perp-y_\perp\|) \le u_h - y_h \\
\label{U_equatN}  \eq&\quad \max(x_h + f(\|u_\perp-x_\perp\|) , y_h + f(\|u_\perp-y_\perp\|) ) \le u_h .
\end{align}
In particular, for $u \in S$, we have 
\begin{equation*}
 u\in U \EQ \max(x_h + f(|u_e-x_e|) , y_h
+f(|u_e - y_e|) ) \le u_h .
\end{equation*}
Because $y\notin U$, the lowest point of $U\cap S$ above $y$ is
\begin{equation*}y' = y_e\, e + y_h' \, h\text{, where } y_h' = x_h+f(y_e-x_e) > y_h.\end{equation*}
Let $S'$ be the rectangle of the points $z\in S$ with
$x_e \le z_e\le y_e$ and $y_h \le z_h \le y'_h$.  As $S'$ is
compact and $U$ is closed, the intersection $U\cap S'$ is compact.  It
is nonempty because $y' \in U\cap S'$.  Therefore, there is
$z\in U\cap S' $ with $z_h \le u_h$ for all $u\in U\cap S'$.
We claim that $z_h \le u_h$ for all $u\in U\cap S$; it suffices to
consider $u\in U\cap S\setminus S'$. In that case, if $x_e \le u_e \le y_e $,
then $z_h \le y_h' < u_h$. If $y_e < u_e$, then $z_h \le y_h' < u_h$
because $f$ is increasing.  The case $u_e < x_e$ is treated in the
same way.  This proves that $z_h \le u_h$ for all $u\in U\cap S$.

As $z$ is a lowest point of $U\cap S$ and $x_e \le z_e \le y_e$,
we have
\begin{equation*} 
 z_h = \max(x_h + f(z_e - x_e), y_h + f(y_e - z_e)) .
\end{equation*}
At first sight, one might expect the two terms of the maximum to be
equal, but this need not be the case because of the semicontinuity of
$f$.  Instead, we claim that
\begin{equation}  \label {leq_f^+}
 z_h \le x_h + f^+(z_e - x_e)
\Land z_h \le y_h + f^+(y_e - z_e) .
\end{equation}
The lefthand inequality is treated first.  If $z_e = y_e$, then
$z = y'$ and $z_h = x_h+f(y_e-x_e)$, which is less than $x_h + f^+(y_e-x_e)$.
Otherwise, it holds that $x_e \le z_e < y_e$.  Assume that
$x_h + f^+(z_e - x_e) < z_h$. Then there is a real number $t$ with
$ z_e < t < y_e $ and $s = x_h + f(t - x_e ) < z_h$.  As $f$ is
increasing, we also have $s' = y_h + f(y_e- t) < z_h$.  If we put
$ s'' = \max(s, s') $, the vector $ u = t\, e + s''\, h$
satisfies $u \in U$ and $u_h = s'' < z_h $, contradicting the
minimality of $z_h$.  This proves the lefthand inequality of (\ref
{leq_f^+}).  The other one follows by symmetry.

It remains to prove that every element $u \in U$ is a right-bound of
$z$.  Let $u\in U$ be given.  As we need to compare the vectors
$u_\perp$ and $z_\perp$, we define $q = u_\perp - z_\perp $.  First assume
that $q=0$.  This implies that $u_\perp = z_\perp = z_e\,e$.  It follows
that $u\in U\cap S$, and hence $z_h \le u_h$, and hence
$ z \preceq_f u$.

It remains to assume that $q\ne 0$.  Two cases are distinguished: $q_e
\ge 0$ or $q_e\le 0$.  Assume $q_e \ge 0$.  Put $ p = z_\perp -
x_\perp$.  Then $(p, q) \ge 0$.  We use \cref {norms_superplus}
with $E := h^\perp $, and $p$ and $q$ as chosen just now.  The relation
$z\preceq_f u$ is proved in
\begin{align*}
  & z \preceq_f u \\
  \eq&\text{\Com\ definition $\preceq_f$ \moc}\\
  & f(\|u_\perp - z_\perp\|) \le u_h - z_h \\
  \follows&\text{\Com\ (\ref {U_equatN}) gives $x_h + f(\|u_\perp - x_\perp\|) \le u_h$ \moc}\\
  & z_h + f(\|u_\perp - z_\perp\|) \le x_h + f(\|u_\perp - x_\perp\|) \\
  \eq&\text{\Com\ choices of $p$ and $q$ \moc}\\
  & z_h + f(\|q\|) \le x_h + f(\|q + p\|) \\
  \follows&\text{\Com\ $q\ne 0$, Lemma \ref {norms_superplus} with $E:=h^\perp$, and choice of $p$ \moc}\\
  & z_h \le x_h + f^+(\|z_\perp -x_\perp \|) \\
  \eq&\text{\Com\ $x_\perp = x_e\, e$, $z_\perp = z_e\, e$, and \cref {leq_f^+} \moc}\\
  &\mathrm{true}\ .
\end{align*}
The case $q_e \le 0$ is treated in the same way with
$p = z_\perp - y _\perp $.
\end{proof} 

\medbreak
The if parts of \cref {thm:epigraph_sponges} are now
obtained by collecting the results. 
Assume that $\preceq_f$ is topologically closed, and that $f$ is square-superadditive (superadditive if $\dim(E)=2$).  Then \cref {thm:everyPairAJoin}
implies that every pair has a join in $(E, \preceq_f)$.  As $f$ is
superadditive and satisfies $f(d)>0$ for all $d>0$, it is increasing.
Therefore, \cref{thm:hDiscriminator} implies that $(E, \preceq_f)$ has a discriminator.
Therefore, \cref{metric2spongeDual} implies that $(E, \preceq)$ is a cc sponge.
This concludes the proof of \cref {thm:epigraph_sponges}.

\section {The hyperbolic sponge}  \label {hyp_sponge}
Let $h$ be a unit vector in a real Hilbert space $E$ with $\dim(E)\ge2$.  Let $ H =
h^\perp $ be the hyperplane orthogonal to $h$ and let $H^+ = \{x\in E
\mid 0 < (h,x) \}$ be the (open) half space in direction $h$. We again have $x_h=(h,x)$ and $x_\perp = x -
(h,x)h$.

It is known that $H^+$ can be considered a model for hyperbolic space \cite[\textsection7]{Cannon1997HG}: the Poincar\'e half-space model. In this model, the distance between two points in $H^+$ is
\begin{equation*}
d_\mathcal{H}(x,y)=\arcosh\left(1+\frac{\|x-y\|^2}{2\,x_h\,y_h}\right)
.\end{equation*}
Where $\arcosh(x)=\ln(x+\sqrt{x^2-1})$.
It can be checked \cite[\textsection12.2.1]{Angulo2014MPU} that for two points $x,y$ such that $x_\perp=y_\perp$, $d_\mathcal{H}(x,y)=|\ln(x_h) - \ln(y_h)|$.
Because of this, we put
\begin{equation*}
h_\mathcal{H}(x)=\ln(x_h)
.\end{equation*}
Since $x\in H^+$, $x_h>0$, and the above is well-defined.
Note that while $H^+$ is an open subset of $E$, $(H^+,d_\mathcal{H})$ is, in fact, a complete metric space.
Furthermore, the metric spaces $(H^+,d_\mathcal{H})$ and $(H^+,d_E)$ with $d_E(x,y)=\|x-y\|$ are homeomorphic, as the identity function is a homeomorphism between the two.

\begin{figure}\centering%
\begin{comment}
\begin{tikzpicture}[inner sep=0pt,minimum size=2mm,scale=0.8]
\draw [fill=lightgray] (3,0) arc (0:180:3);
\draw [dotted] (2,0) arc (0:138.59:4);
\draw [draw=gray] (-5,0) -- (5,0) node [right] {$H$};
\draw [dashed] (0,3) -- (0,0);
\node at (0,3) [circle,fill,label=above:$y$] {};
\node at (0,0) [circle,fill,label=below:$\vphantom{y^1}y_\perp$] {};
\node at (2.3,1) [circle,fill,label=above:$\vphantom{y}x$] {};
\node at (-2,4) [circle,fill,label=above:$\vphantom{y}z$] {};
\end{tikzpicture}
\end{comment}
\includegraphics{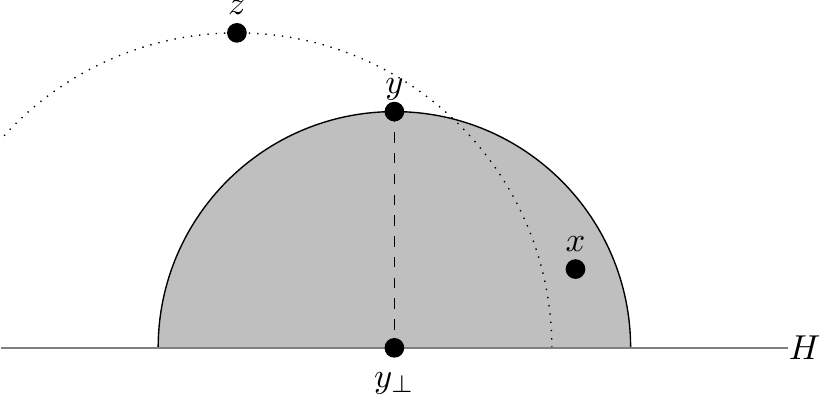}%
\caption{Illustration of the hyperbolic orientation in the Poincar\'e half-plane model. We have $x\preceq y$, as well as $y\preceq z$, but $x$ and $z$ are incomparable. Note how the set of left bounds of $y$ is half of a closed disk in this model, and that only points strictly above $H$ correspond to points in hyperbolic space.}
\label{fig:hyperbolicOrientation}
\end{figure}

Relation $\preceq$ on $H^+$ is defined by 
\begin{equation*}
x\preceq y \EQ \|x-y_\perp\| \le y_h .
\end{equation*}
Note that $x \preceq x$ holds because $ \|x - x_\perp\| = x_h $.
We observe that the above definition corresponds to saying that $x\preceq y$ if and only if $y$ lies between $x$ and the highest point of the geodesic through $x$ and $y$ in the half-space model of hyperbolic space \cite[Thm.\@ 9.3]{Cannon1997HG}; this is the converse of the original formulation \cite[\textsection12.4.4]{Angulo2014MPU}, but for the current exposition it was much more convenient to use the convention that ``higher'' values are larger.

\begin{lemma} \label {reflexU}
  Let $x \preceq y$ and $x\ne y$. Then $ h_\mathcal{H}(x)< h_\mathcal{H}(y) $.
\end{lemma}

\begin{proof}
  As $x_h$ is the height of $x$ above $H$, and $y_\perp\in H$, we have
  $x_h \le \|x-y_\perp\|$, with equality if and only if $x_\perp =
  y_\perp$.  On the other hand, $x\preceq y$ implies $\|x-y_\perp\| \le
  y_h$. Considering that the logarithm is increasing on the positive reals, the inequality follows, unless $x_\perp=y_\perp$.

  Therefore, assume that $x_\perp=y_\perp$.  We then have $x=x_h h+x_\perp$ and $y=y_h h+x_\perp$.
  Given that $x\preceq y$, we see that $\|x-y_\perp\|=x_h\le y_h$.
  As $x\ne y$, we see that $x_h<y_h$, and thus $h_\mathcal{H}(x) < h_\mathcal{H}(y)$.
\end{proof}

\begin{corollary} \label {cor1}
Relation $\preceq$ is an acyclic orientation on $H^+$. 
\end{corollary}

\begin{proposition}  \label {hyper_bounds}
(a) Every finite subset $P$ of $H^+$ has an right bound in $(E,
\preceq)$. \\
(b) If a pair $x$, $y\in H^+$ has a left bound if and only if $\|x_\perp -
y_\perp\| < x_h+y_h$. 
\end{proposition}

\begin{proof} 
  (a) Choose $\lambda > 0$ with $\|p\| \le \lambda$ for all $p\in P$.
  Then $y = \lambda \, h $ is a right bound of $P$, because
  $\| p - y_\perp \| = \|p\| \le \lambda = y_h $ for every $p\in P$.

  (b) Let $p$ be a left bound of $x$ and $y$. Then $p\in H^+$.
  Therefore $p$ is not on the line segment between $x_\perp$ and
  $y_\perp$.  This implies
  $\|x_\perp - y_\perp \| < \|x_\perp - p\| + \|p - y_\perp\| \le x_h +
  y_h$.  The converse follows from considering when two (hemi)spheres
  overlap.
\end{proof} 

\begin{remark}
  It follows that there are pairs without a left bound.  For example,
  let $e$ be a unit vector orthogonal to $h$, take $x = h$ and
  $y = h + 2 \, e$.  The pair $x$, $y$ has no left bound because
  $\|x_\perp-y_\perp\| = 2$ and $x_h=y_h = 1$.
 This implies that inversion of the orientation gives a completely
 different oriented set. 
\end{remark}

\begin{lemma}\label{thm:hyperbolicOrientationClosed}
The relation $\preceq$ on the complete metric space $(H^+,d_\mathcal{H})$ is closed.
\end{lemma}
\begin{proof}
We first show that the relation is closed on the metric space $(H^+,d_E)$ with the Euclidean metric.
To this end, consider the function $f:H^+\times H^+\rightarrow\IR$ given by $f(x,y)=y_h-\|x-y_\perp\|$.
As $f$ is continuous under the Euclidean metric, and $\preceq$ is the preimage of the closed set $\{t\mid t\ge0\}$, $\preceq$ is closed.
Since $(H^+,d_\mathcal{H})$ and $(H^+,d_E)$ are homeomorphic, $\preceq$ is closed in $(H^+,d_\mathcal{H})$ as well.
\end{proof}

\begin{lemma}\label{thm:hyperbolicDiscriminator}
On the complete metric space $(H^+,d_\mathcal{H})$, $h_\mathcal{H}$ is a discriminator.
\end{lemma}
\begin{proof}
In order to prove that function $h_\mathcal{H} $ is a discriminator, we
try, given $y\in H^+$ and $\delta > 0$, to bound the distance
$d_\mathcal{H}(x, y)$ for all vectors $x$ in the set 
\begin{multline*}
L_\delta(y) = \{x\in H^+ \mid x \preceq y \Land h_\mathcal{H}(y)
  < h_\mathcal{H} (x) + \delta \} \\
= \{ x\in H^+ \mid \|x-y_\perp\| \le y_h \Land \ln (y_h) <
  \ln(x_h) + \delta \} .
\end{multline*}
In view of the formula for $d_\mathcal{H}(x, y)$, the maximal value of
this distance is obtained by maximizing $\|x-y\|$ and minimizing
$x_h$.  The maximal distance is therefore reached when
$\|x-y_\perp\| = y_h $ and $\ln (y_h) = \ln(x_h) + \delta$.  This
maximal distance is not reached in $L_\delta(y)$, however, but only on
its boundary.  In any case, such vectors $x$ give the least upper
bound of the distance.  If we write $x_u = \| x_\perp-y_\perp\|$, these two
equations become $x_h^2 + x_u^2 = y_h^2$ and $ y_h = x_he^\delta $.
After some calculation, one finds that
$d_\mathcal{H}(x, y) = \arcosh(e^\delta)$ holds because
\begin{equation*}
1+\frac{\|x-y\|^2}{2\,x_h\,y_h}
= \frac{2\,x_h\,y_h+(x_h-y_h)^2+x_u^2}{2\,x_h\,y_h}
= \frac{2\,y_h^2}{2\,x_h\,y_h}
= e^\delta
.\end{equation*}
Using continuity of the $\arcosh$
function for $\delta \downarrow 0$, one finds that $h_\mathcal{H}$ is a
discriminator.
\end{proof}

\begin{theorem} \label {jointhm}
The pair $(H^+, \preceq)$ is a cc sponge. 
\end{theorem}
\begin{proof}
We have already shown that $\preceq$ is closed, and that $h_\mathcal{H}$ is a discriminator. It is also clear that $h_\mathcal{H}$ is continuous. Thus, if we can show that every left-bounded pair in $H^+$ has a meet, we can apply \cref{metric2sponge}.
Now, let $x$, $y$ be a left-bounded pair in $H^+$.  If $x$ and $y$ are
  comparable, one of them is the meet.  We may therefore assume that
  $x$ and $y$ are not comparable.  It follows that
  $x_\perp \ne y_\perp$.

  Let $e$ be a unit vector pointing from $x_\perp$ to $y_\perp$.  Let
  $a = \|y_\perp-x_\perp\|$.  Then $a > 0$ and
  $y_\perp = x_\perp + a\, e$.  \Cref  {hyper_bounds}(b)
  implies that $ a < x_h+y_h$. 

  Let $S$ be the plane that contains the points $x$, $x_\perp$, $y$,
  $y_\perp$.  Let $S^+ = S \cap H^+$.  In the halfplane $S^+$, the set
  of left bounds of $x$ is the half disk with center $x_\perp$ and
  radius $x_h$; similarly for $y$.  These half disks intersect because
  $ a < x_h+y_h$.  As $x$ and $y$ are not comparable, it is not the
  case that one of the half disks is contained in the other.
  Therefore, the corresponding circles meet in two points, one of
  which is in $ S^+$.  Assume that the circles meet in $z \in S^+$.
  As $z$ is contained in both
  half disks, it is a left bound of both $x$ and $y$.
  We claim that $z$ is the meet of $x$ and $y$.

  The projection $z_\perp$ is on the line through $x_\perp$ and
  $y_\perp$, and can therefore be written
  $z_\perp = x_\perp + b\, e $.  As $x$ and $y$ are not comparable,
  we have $0 < b < a$.  Let $c = a - b$.  Then
  $y_\perp = z_\perp + c\, e$.   It follows that
\begin{equation} \label {oppspheres1}
\left.\begin{aligned}
b^2 + z_h^2 &= x_h^2\\
c^2 + z_h^2 &= y_h^2\\
b + c &= a .
\end{aligned}\;\right\}
\end{equation}

In order to prove that $z$ is the meet of $x$ and $y$, it remains to
prove that any left bound $u$ of $x$ and $y$ is a left bound of $z$.
Let $z_\perp + t\, e$ be the orthogonal projection of $u$ onto the line through
$x_\perp$ and $y_\perp$, and let $s$ be the distance of $u$ to this
line.  We now have
\begin{align*}
  & u \preceq x \Land u \preceq y  \\
  \eq&\text{\Com\ definition \moc}\\
  &  (t+b)^2 + s^2 \le x_h^2 \Land (t - c)^2 + s^2 \le y_h^2 \\
  \eq&\text{\Com\ \Cref {oppspheres1} \moc}\\
  & t^2 + 2 b t + s^2 \le z_h^2 \Land t^2 - 2ct +s^2 \le z_h^2 \\
  \Rightarrow&\text{\Com\ $b > 0$ and $c > 0$, and hence $2bt \ge 0$ or $2ct \le 0$  \moc}\\
  & t^2 + s^2 \le z_h^2 \\
  \eq&\text{\Com\ definition \moc}\\
  & u \preceq z .
\end{align*}
This proves that $z$ is the meet of $x$ and $y$. Considering \cref{thm:hyperbolicOrientationClosed,thm:hyperbolicDiscriminator}, and observing that $h_\mathcal{H}$ is continuous, we can now apply \cref{metric2sponge} to conclude that $(H^+,\preceq)$ is a cc sponge.
\end{proof}

\section {The geometry of the various sponges}

In order to compare the various sponges we constructed, it is useful
to examine the left cones $L(x)$ and the right cones $R(x)$ of elements in the different sponges.

In a sponge group, all left and right cones are isomorphic because 
$R(x) = x + R(0)$ and $L(x) = x + L(0)$ and $ L(0) = -R(0)$. 

In the inner-product sponge of \cref {ip_sponge}, every right
cone $R(x)$ for $x\ne 0$, is a half space, while the left cone $L(x)$ is
the ball centered at $\half x$ with radius $\half\|x\|$.  In the hyperbolic sponge of \cref
{hyp_sponge}, every left cone is a half ball, while the right cone is
bounded by a component of a hyperboloid.

The inner-product sponge has precisely one left-extreme point,
viz. the origin of the space, and no right-extreme points.
The hyperbolic sponge has no extreme points.

In the inner-product sponge, every nonempty subset has a meet, which
can be the origin.  In the hyperbolic sponge, every finite or bounded
subset has a join.

In the inner-product sponge, the right cones $R(x)$ and $R(y)$ are
disjoint if and only if $x \ne 0$ and $y = \lambda x$ for some
$\lambda < 0$.  In the hyperbolic sponge, the left cones $L(x)$ and
$L(y)$ are disjoint if and only if
$x_h + y_h \le \|x_\perp - y_\perp\|$.

\providecommand{\doi}[1]{doi:\discretionary{}{}{}\href{http://dx.doi.org/#1}{\urlstyle{same}\nolinkurl{#1}}}
%\bibliographystyle{../../../../plainnaturl}
%\bibliography{refs,jaspervdg-nourls-abbrv,group-nourls-abbrv}

\end{document}

%% file: latexmacrosG.tex
\usepackage{latexsym}
\usepackage{amssymb}
\usepackage[sort&compress,capitalize]{cleveref}

\newenvironment{tab}{\begin{tabbing}
MMMMM\=aaa\=aaa\=aaa\=aaa\=aaa\=aaa\= \kill}{\end{tabbing}}

\makeatletter
\def\refmystepcounter#1{\stepcounter{#1}\protect\gdef 
\@currentlabel {\csname p@#1\endcsname \csname 
the#1\endcsname}}
\makeatother
\newcounter {tabnr}

\def\Nat   {\mbox{$\mathbb{N}$}}

\def\IR    {\mbox{$\mathbb{R}$}}

\def\IZ    {\mbox{$\mathbb{Z}$}}

\def\sqleq {\sqsubseteq}

\def\Com   {\mbox {$\quad\{ $}}

\def\moc   {\mbox {$ \}\; $}}

\def\eps   {{\mbox{$\varepsilon$}}}
\def\phi   {{\mbox{$\varphi$}}}

\def\follows{\mbox{$\Leftarrow $}}

\def\EQ     {\mbox{\quad$\equiv$\quad}}
\def\eq     {\mbox{$\equiv\,$}}

\def\Land   {\mbox{ $\;\land\;$ }}

\def\half   {\mbox{$\frac{1}{2}$}}